\documentclass[a4paper,leqno]{amsart}
\usepackage{graphicx, amsmath, amssymb, amsthm, hyperref, verbatim, multicol, mathrsfs, tikz, caption, subcaption,float}
\usepackage{stmaryrd}
\usepackage{mathtools}  
\usepackage{cleveref}   
\usepackage{todonotes}
\usepackage{tikz-cd}
\usepackage{physics}
\usepackage{enumitem}   
\usepackage{wrapfig, blindtext}
\newtheorem{Theorem}{Theorem}[section]
\newtheorem{Lemma}[Theorem]{Lemma}
\newtheorem{Main theorem}[Theorem]{Main theorem}
\newtheorem{Corollary}[Theorem]{Corollary}

\newtheorem{Key lemma}[Theorem]{Key lemma}

\newtheorem{Remark}[Theorem]{Remark}

\newtheorem{Proposition}[Theorem]{Proposition}
\newtheorem{Question}[Theorem]{Question}

\theoremstyle{definition}

\theoremstyle{remark}

\begin{document}

\title{Curvature bounds on length-minimizing discs}

\author{Alexander Lytchak}
	\address{Fakultät für Mathematik, Karlsruher Institut für Technologie, Englerstraße 2, 76131 Karlsruhe, Germany.}
	\email{alexander.lytchak@kit.edu}

\author{Sophia Wagner}
	\address{Fakultät für Mathematik, Karlsruher Institut für Technologie, Englerstraße 2, 76131 Karlsruhe, Germany.}
	\email{sophia.wagner@kit.edu}

\thanks{The authors are partially supported by SFB/TRR 191 ``Symplectic structures in Geometry, Algebra and Dynamics, funded by the DFG" \newline {\it 2010 Mathematics Subject Classification.}  53C21, 53C23, 53C42, 53C43.
	\newline {\it Keywords.}  Alexandrov geometry, harmonic surfaces, ruled surfaces}
\maketitle
\begin{abstract}
We show that a length-minimizing  disk inherites the upper curvature bound of the target. 
As a consequence we prove that harmonic discs and ruled discs
inherit the upper curvature bound from the ambient space.
\end{abstract} 

\section{Introduction}
 We  extend the  results of \cite{Petold},\cite{petrunin2019metric} to the case of curvature bounds different from $0$ and  provide full proofs of some consequences of 
the main theorems, which are mentioned  in \cite{petrunin2019metric} in the case of 
 $\textsc{CAT}(0)$  spaces.  

We need some notation, in order to state the main result. 
Let  $f:\mathbb D \to Y$ be a continuous map from the closed unit disc $\mathbb D$ into a $\textsc{CAT}(\kappa)$ space $Y$.  For $x,z\in \mathbb D$  we define the \emph{length pseudodistance}  $\langle x-z\rangle _f \in [0, \infty]$ \emph{induced by} $f$ as
\begin{equation} \label{eq: lengthps}
 \langle x -z \rangle _f:= \inf _{\gamma} \ell_Y(f\circ \gamma) \in [0,\infty]\;,
\end{equation} 
where $\ell_Y(\eta)$ denotes the   length of a curve $\eta$  in the metric space $Y$, and the infimum  in \eqref{eq: lengthps} is taken over all curves $\gamma$ in $\mathbb D$ connecting $x$ and $z$.

We say that $f$ is  \emph{length-connected}  if the value  $\langle x -z \rangle _f$ is 
finite, for all $x,z\in \mathbb D$. In this case, the length pseudodistance defines a metric space  by identifying points with length pseudodistance $0$. This  space is denoted by
$\langle \mathbb D \rangle _f$ and is called the \emph{length metric space induced by} $f$.   

We say that
$f$ is \emph{length-continuous}  if the canonical projection $\hat \pi _f :\mathbb D\to \langle \mathbb D \rangle _f$ is continuous.    Length-continuity always holds if $f$ is a composition of a homeomorphism $\Phi: \mathbb D\to \mathbb D$ and a Lipschitz map $g:\mathbb D\to Y$.

A continuous map $f:\mathbb D \to Y$  is \emph{length-minimizing}
if  for all  continuous $g:\mathbb {D} \rightarrow Y$ with $g|_{\mathbb S^1} =f|_{\mathbb S^1}$ and  
\begin{equation} \label{eq: lengthmin} 
\ell_Y(g\circ \gamma ) \leq \ell_Y (f\circ \gamma)\,,
\end{equation}
for  all curves $\gamma$  in $\mathbb D$, equality must hold   
in \eqref{eq: lengthmin}, for all curves $\gamma$.

Now we can state our main result:

\begin{Theorem}
\label{Theorem:main}
Let $f:\mathbb{D}\rightarrow Y$ be a length-continuous  and length-minimizing map from the disc $\mathbb D$ to  a\; $\textsc{CAT}(\kappa)$ space  $Y$.
Then  $\mathbb D$ with the length metric induced by $f$ is a\; $\textsc{CAT}(\kappa)$ space.
\end{Theorem}

We use this result in order to find upper curvature bounds on discs satisfying more common minimality assumptions. For the first consequence, we assume some familiarity with the notion of harmonic maps
due to   Korevaar--Schoen \cite{korevaar1993sobolev}.

\begin{Corollary} \label{cor:harm}
Let $Y$ be a \textsc{CAT}$(\kappa )$ space and  let the continuous map $f:\mathbb{D}\rightarrow Y$ be  harmonic. If the boundary curve $f:\mathbb S^1 \to X$ has finite length
then $\mathbb{D}$ with the  length metric induced by $f$ is a \textsc{CAT}$(\kappa)$ space.
\end{Corollary}

If $f$ is \emph{conformal}, the conclusion of Corollary \ref{cor:harm} appears in \cite{LW-isop} and of a  closely related statement in  \cite{Mese}. For non-conformal harmonic discs it seems barely possible  to prove
Corollary \ref{cor:harm} by purely analytic means.

Another application concerns ruled discs:

\begin{Corollary} \label{cor: ruled}
Let $\eta_0, \eta _1 :[0,1]\to Y$ be rectifiable curves in a  \textsc{CAT}$(\kappa )$ space $Y$. If $\kappa >0$, we assume that the distance between $\eta _0(a)$ and $\eta _1(a)$ is less than $\frac {\pi} {\sqrt {\kappa}}$, for all $a\in [0,1]$. For any $a>0$, consider the geodesic $\gamma _a:[0,1] \to Y$ between  $\eta _0(a)$ and $\eta _1(a)$ parametrized proportionally to arclength. 
Then $\mathcal  [0,1]\times [0,1]$ with the  length metric induced by the map  $ f(a,t):=\gamma _a(t)$ is  \textsc{CAT}$(\kappa )$.
\end{Corollary}

The  statement about the inheritance of upper curvature bounds by ruled discs appeared in somewhat different generality   with a sketchy proof  in \cite{alexandrov1957uber}, the paper that gave birth to the theory of \textsc{CAT}$(\kappa )$ spaces. Missing details in  Alexandrovs proof were recently provided 
 by Nagano--Shioya--Yamaguchi in \cite{nagano2021two}.

It seems possible to extend our proof  to the case of non-rectifiably boundary curves 
and to dispense of the assumption on the parametrization of the geodesics. However, it would require rather  technical considerations. The generality we have chosen is sufficient for most applications and follows directly from Theorem \ref{Theorem:main}.

We finish the introduction with several comments.
 
\begin{itemize}
 \item  Our notion of length-minimality corresponds to the notion called \emph{metric-minimality}  in \cite{Petold}.   In \cite{petrunin2019metric} a new, stronger and less natural notion of metric-minimality was introduced. With this new stronger notion, a version of our main theorem for  \textsc{CAT}$(0 )$ spaces was proved only under the assumption of length-connectedness of $f$, instead of the stronger length-continuity.  We have only been able to verify the  validity of this stronger metric-minimality assumption  in  cases  covered by our Theorem \ref{Theorem:main}. In order to find a clearer and better comprehensible way to   Corollaries \ref{cor:harm},
\ref{cor: ruled}, we have   decided to work with the original,  more natural notion of metric-minimality used in \cite{Petold}.  In order to distinguish it from metric-minimality used in 
 \cite{petrunin2019metric}, we have given it a different name.

\item In the last section we formulate several questions concerning generalizations of our main results.

\item
The length metrics arising in Theorem \ref{Theorem:main} and, therefore, in 
Corollaries \ref{cor:harm}, \ref{cor: ruled} are homeomorphic to deformation retracts of $\mathbb D$.

\item
A general existence result in Section \ref{sec: exist} shows that 
length-minimizing discs are abundant,  beyond harmonic and ruled discs.

\item
The main argument  closely  follows  \cite{petrunin2019metric}.   A central  technical step  in the proof is  the reduction of the theorem to the case, where the map $f$ has totally disconnected fibers.   In order to achieve this  reduction, another metric induced by $f$ on the disc is investigated, the so-called \emph{connecting pseudometric}.  It 
 has better topological and metric properties than the more natural \emph{induced length metric}.  Roughly speaking, while the induced length metric  collapses curves sent by $f$ to a point, the connecting pseudometric also collapses \emph{pseudocurves} sent by $f$ to a point.  In  general, the precise relation 
between the two metrics seems to be  rather complicated. However, as verified in \cite{petrunin2019metric}, the length-minimality together with the length-connectedness, imply that the  induced length metric is just the intrinsic metric induced by the connecting pseudometric. 

\item
The heart of the proof is  an approximation of the  induced length metric by polyhedral metrics, whose restriction to the $1$-skeleton is length-minimizing. 
This idea presented  in \cite{Petold} works in our setting  with minor modifications.
The case of $\kappa >0$ requires some additional considerations,  due to the absence of the   theorem of  Cartan--Hadamard.
\end{itemize}

\subsection*{Acknowledgments}
We thank  Paul Creutz, Anton Petrunin and Stephan Stadler for  helpful discussions and comments.

\section{Preliminaries}
\label{sec:prelis}

\subsection{Topology}
A \emph{curve} in a topological space $X$ is a continuous map $\gamma : I\rightarrow X$, where $I\subset\mathbb{R}$ is an interval. A \emph{Jordan curve} in $X$ is a subset homeomorphic to $\mathbb S^1$.

A topological space $X$ is a \emph{Peano continuum} if $X$ is metrizable, compact, connected and locally connected.  In this case $X$ is arcwise connected.

A \emph{disc retract} is a compact space homeomorphic to a subset  $X'$
 of the closed unit  disc $\mathbb D$  such that $X'$ is  a homotopy retract of $\mathbb D$.  
 A Peano continuum is a disc retract if and  only if  it is  homeomorphic to a non-separating subset of the plane, \cite[p. 27]{Morrey}. 
We will only need the following  statement about disc retracts, a variant of a classical theorem of Moore, \cite{Moore}, \cite[Proposition 3.3]{petrunin2019metric}: 

 If $f:\mathbb D\to X$
is a surjective, continuous map, such that all fibers are connected and non-separating subsets of $\mathbb R^2$, then $X$ is a disc retract.

\begin{Remark}  The following properties will not be used below. The  reader might find them helpful, since disc retracts are central objects  of this paper. 

A space $X$ is a disc retract if and only if there exists a closed curve $\gamma :\mathbb S^1\to X$,  such that the mapping cylinder $([0,1]\times \mathbb S^1) \cup _{\gamma} X$ is homeomorphic to $\mathbb D$,  \cite{petrunin2019metric}, \cite{creutz2021space}. The image of  $\gamma$  is  the \emph{boundary} of $X$, defined  as the set of points at which $X$ is not a $2$-manifold. 

Any maximal subset of a disc retract  $X$ which is not a point and does not contain cut points is  a closed disc \cite[p.27]{Morrey}, \cite{petrunin2019metric}. 
The number of these \emph{cyclic components} of $X$ is at most countable.

By approximation, it follows that any disc retract admits a CAT$(-1)$ metric.  
The space of (isometry classes of) disc retracts with a geodesic metric,  uniformly bounded $\mathcal H^2$-measure and  $\mathcal H^1$-measure of the boundary and a \emph{quadratic isoperimetric inequality} is compact in the Gromov--Hausdorff topology,  \cite{creutz2021space}. 
In particular, the space of isometry classes of CAT$(\kappa)$ disc retracts with boundary of   $\mathcal H^1$-measure at most $r< \frac {2\pi} {\sqrt {\kappa}}$ is compact in the Gromov--Hausdorff topology. 
 \end{Remark}

\subsection{Metric geometry}
We stick to the notations and conventions used in \cite{petrunin2019metric}, and refer to  \cite{burago2022course} and \cite{petruninmetric} for introductions to metric geometry.

We denote the distance in a metric space $Y$ by $|\ast -\ast |_Y=\langle\ast -\ast\rangle =\langle\ast -\ast\rangle_Y$.

The length of a curve $\gamma$ in $Y$ will be denoted by $\ell_Y(\gamma)=\ell(\gamma)$. 
A \emph{geodesic} is a curve $\gamma$ connecting points $y_1,y_2$ in $Y$, such that the  length of $\gamma$ equals $\langle y_1-y_2 \rangle _Y$.  We make no a priori assumption on the parametrization of a geodesic.

The space $Y$ is a \emph{length space} if the distance between any pair of its points equals the infimum of lengths of curves connecting these points.   The space is a \emph{geodesic space} if any pair of points is connected by a geodesic.

A \emph{pseudometric space} is a metric space in which the distance can also assume values $0$ and $ \infty$.  If the value $\infty$ is not assumed by  a pseudometric on  $Y$, we identify subsets of points of $Y$  with pseudo-distance equal to $0$ and obtain the \emph{induced metric space}
\cite[Section 1.C]{petruninmetric}.

We assume some familiarity with the Gromov--Hausdorff convergence and, for readers interested in the non-locally  compact targets $Y$ in our main results,
 with ultralimits.  We refer  to \cite[Sections 5,6]{petruninmetric} and \cite[Section I.5]{Bridson}.

\subsection{\textnormal{CAT}$(\kappa )$ spaces}
We assume familiarity with properties of \textnormal{CAT}$(\kappa )$ spaces,  the reader is refered to \cite{burago2022course}, \cite{Bridson}, \cite{alexander2014alexandrov} and \cite{ballmann2012lectures}. We will stick to the convention that \textnormal{CAT}$(\kappa )$ spaces are complete and geodesic.

By $M_{\kappa}$ we denote the model surface of constant curvature $\kappa$ and by $R_{\kappa}$ its diameter. Thus,  $R_{\kappa}=\frac{\pi}{\sqrt{\kappa}}$ if $\kappa >0$ and $R_{\kappa}=\infty$ if $\kappa\leq 0$.

\section{Metrics induced by maps}
We recall   some notions and  notations introduced  explicitely or implicitely  in \cite[Section 2]{petrunin2019metric}.   All statements in this section are proved there.
 
\subsection{Length metric induced  by a map}\label{Metrics induced by maps}

Let $X$ be a topological space and $Y$ a metric space. A continuous map $f:X\rightarrow Y$ induces a pseudometric on $X$ by 
\begin{center}
$\langle x-z\rangle_f =\inf\lbrace \ell_Y(f\circ\gamma) : \gamma\;\; \text{is a curve in $X$ joining $x$ to $z$}\rbrace$.
\end{center} 
This pseudometric is called the \textit{length pseudometric on $X$ induced by $f$}.

As in the introduction, we call a continuous map   $f:X\to Y $  is \emph{length-connected}     if this length pseudometric is finite for all pairs of points. Thus, if  any $x,z \in X$  are connected in $X$ by a curve $\gamma$ whose image $f\circ \gamma$ has finite length in $Y$.

If $f:X\to Y$ is length-connected  then 
 the metric space arising from the pseudometric
  $\langle x-z\rangle_f $  by identifying points at pseudistance $0$ will be called \emph{the length metric space induced by  $f$}. We denote it by 
$\langle X\rangle_f$.

The canonical surjective projection from   $X$ onto  $\langle X\rangle_f$ will be denoted by 
\begin{center}
$\hat{\pi}_f :X\rightarrow\langle X\rangle_f$.
\end{center}

We will say that a continuous map $f:X\to Y$ is \emph{length-continuous}   if $f$  is length-connected and the induced projection  $\hat{\pi}_f :X\rightarrow\langle X\rangle_f$ is continuous.

For  a compact metric space $X$, a map $f:X \to Y$ is length-continuous if and only if for any $\varepsilon >0$ there exists $\delta >0$, such that any pair of points $x,z\in X$ with
$\langle x -z \rangle _X<\delta$ is connected by some curve $\gamma $ in $X$ with
$\ell_Y(f\circ \gamma) <\varepsilon$.

If $X$ is a length metric space and $f:X\to Y$ is locally Lipschitz continuous then 
$f$ is length-continuous.

For any length-connected $f:X\to Y$,  there exists a unique $1$-Lipschitz map
$$\hat{f} : \langle X\rangle_f \rightarrow Y\,,$$ 
such that
$\hat{f} \circ \hat{\pi}_f =f$.  The following  is stated between the lines in \cite{petrunin2019metric}.  

\begin{Lemma} \label{lem: zusatz}
Let $f:X\to Y$ be length-connected.  Let $\gamma$ be a curve in $X$.  
The curve $f\circ \gamma$ has finite length in $Y$ if and only if  $\hat {\pi} _f  \circ \gamma$ is a curve of finite length in $ \langle X\rangle _f$. In this case
 $$\ell _Y (f\circ \gamma)= \ell _{\langle X\rangle _f} (\hat {\pi}  _f \circ \gamma).$$
\end{Lemma}

\begin{proof}
If $\hat {\pi} _f  \circ \gamma$ is a curve  of finite length,
then $\ell _Y (f\circ \gamma) \leq  \ell _{\langle X\rangle _f} (\hat {\pi}  _f \circ \gamma)$, since $\hat f$ is $1$-Lipschitz.

Assume now that $f\circ \gamma$ is of finite length. Then, for any $t,s$ in the interval $I$
of definition of $\gamma$, we have
$$   \langle \gamma (t)-\gamma (s)\rangle _f=\langle (\hat \pi _f  \circ \gamma) (t) -(\hat \pi _f  \circ \gamma (s) \rangle _{\langle X\rangle _f} \leq \ell (f\circ \gamma |_{[t,s]})\,.$$    

Since the length of $f\circ \gamma$ on small intervals around a fixed point $t$ goes to $0$ with the length of the interval, 
 $\hat \pi _f \circ \gamma$ is continuous. Applying 
the above inequality to arbitrary partitions of $I$, we deduce 
$ \ell _{\langle X\rangle _f} (\hat {\pi}  _f \circ \gamma ) \leq 
\ell _Y (f\circ \gamma)$.
\end{proof}

\subsection{Connecting pseudometric}
Another pseudometric on $X$ associated with a continuous map  $f:X\rightarrow Y$ is 
the \textit{connecting pseudometric} $|\ast -\ast |_f$  defined as 
\begin{center}
$| x-z|_f =\inf\lbrace\text{diam}\; f(C): C\subset X\; \text{connected and}\; x,z\in C \rbrace$.
\end{center}
Whenever the connecting pseudometric assumes only finite values, we consider
the associated metric spaces and denote it by 
$\big| X\big|_f$.

In this case we have the canonical projection map, denoted by  
\begin{center}
$\bar{\pi}_f : X\rightarrow\big| X\big|_f$.
\end{center}
Moreover, there exist a uniquely defined $1$-Lipschitz map
\begin{align*}
\bar{f} : \big|X\big|_f &\rightarrow Y,
\end{align*}
such that 
$\bar f \circ  {\pi}_f =f\;.$

If  $f:X\to Y$ is length-connected then the connecting pseudometric assumes only finite values and  there exists a uniquely defined, surjective $1$-Lipschitz map
$$\tau_f :\langle X\rangle_f \rightarrow \big| X\big|_f,$$
such that 
$$\bar {\pi} _f  = \tau _f \circ \hat {\pi} _f\,.$$

\subsection{Basic properties} Recall that a map $f:X\to Y$ between topological spaces  is called 
\emph{monotone} (respectively, \emph{light}) if any fiber of $f$ is connected (respectively, totally disconnected).
We recall   from  \cite[Section 2]{petrunin2019metric}:

\begin{Lemma}\label{monotone}
Let $X$ be a Peano continuum and  $Y$ a metric space. Let $f:X\rightarrow Y$ be continuous. Then
\begin{enumerate}
\item[(a)] The map $\bar{\pi}_f :X\rightarrow\big| X\big|_f$ is continuous. Hence,
 $\big| X\big|_f$ is a Peano continuum.

\item[(b)]  
$\ell_{|X|_f} (\gamma )=\ell _Y(\bar{f}\circ\gamma )$, for every curve $\gamma$ in $\big| X \big|_f$.

\item[(c)]
The map $\bar{\pi}_f:X\rightarrow\big| X\big|_f$ is monotone.

\item[(d)] The map $\bar{f}:\big|X\big|_f\rightarrow Y$ is light. 
\end{enumerate}
\end{Lemma}

Lemma \ref{monotone}(b)  implies that $\big| X\big|_f$ with the induced length metric is isometric to $\big| X\big|_f$ with the length metric induced by $\bar{f}$.

From    Lemma \ref{lem: zusatz} and  Lemma \ref{monotone} we obtain:

\begin{Lemma}\label{tau}
Let $X$ be a Peano continuum, $Y$ a metric space and $f:X\to Y$ be 
 length-connected.   Then, for all curves $\gamma$ in $X$,
$$\ell_Y(f\circ \gamma)=\ell_{|X|_f} 
(\bar {\pi} _f \circ \gamma)=\ell_{|X|_f}(\tau _f \circ \hat {\pi} _f \circ \gamma)\;.$$

If, in addition, $\ell _Y(f\circ \gamma)$ is  finite or if $f$ is length-continuous
 then  $\hat \pi _f \circ \gamma$
is continuous and  $$\ell_Y(f\circ \gamma)= \ell_{\langle X\rangle _f} (\hat {\pi} _f \circ \gamma)\,.$$
\end{Lemma}

For a Peano continuum $X$ and a length-continuous  map $f:X\to Y$, 
the natural $1$-Lipschitz map $\tau_f:\langle X\rangle_f \rightarrow\big| X\big|_f$ may be  non-injective   \cite[Example 4.2]{Petruninint}.  Since every connected subset of a finite graph is arcwise connected, this pathology cannot occur if $X$ is a finite, connected graph.   The following result proven in \cite[Lemma 3.3, Proposition  9.3]{petrunin2019metric} is much less trivial:

\begin{Lemma} \label{lem: banah}
Let $f:\mathbb D\to Y$ be length-connected and let any  fiber $f^{-1} (y)$ be a non-separating subset in  $\mathbb R^2$. Then $|\mathbb D|_f$ is a disc retract.  If, in addition, $f$ is length-continuous then the map $\tau_f:\langle \mathbb D\rangle _f \to | \mathbb D|_f$ is a homeomorphism, which  preserves the length of all 
curves in $\langle \mathbb D\rangle _f $. 
\end{Lemma}


\section{Length-minimizing maps}\label{MMDBP}
\label{sec:minimizing}

\subsection{Definition and first properties}\label{Definition}
Let $X$ be a topological space, $Z$ a closed subset of $X$. Let $f:X\to Y$ be a continuous map into a metric space.  For another map $g:X\to Y$ we will write
$f\unrhd g$  (rel $Z$) if $f$ und $g$ coincide on $Z$ and  
\begin{equation} \label{eq:comp} 
\ell_Y(f\circ\gamma )\geqslant \ell_Y(g\circ\gamma )\,,  \; \; \text{for every curve} \; \; \gamma  \;\;  \text{in}  \;\; X.
\end{equation}
 We will call $f$  \textit{length-minimizing relative to} $Z$  if for every map
$f\unrhd g$  (rel $Z$) equality holds in \eqref{eq:comp}, for every curve $\gamma$.

If $X$ is arcwise connected and $Z$ is empty then  length-minimizing maps
relative to $Z$ are constant. From now on, we will always assume that $Z$ is not empty.

\begin{Lemma} \label{lem: trivial}
Let $X$ be a topological space, $Y$ be a metric space  and $Z\subset X$ be  closed. Let $f:X\to Y$ be length-connected.   A continuous map $g:X\to Y$ satisfies 
$f\unrhd g$  (rel. $Z$)   if and only if $g$ is length-connected, coincides with 
$f$ on $Z$ and 
$$\langle x_1 -x_2\rangle _f \geq \langle x_1-x_2 \rangle _g\;,$$
for all $x_1,x_2\in X$.
\end{Lemma}

\begin{proof}
Assume $f\unrhd g$ (rel. $Z$). By definition, $f$ and $g$ coincide on $Z$.  For any 
$x_1,x_2\in X$, we find a curve $\gamma$ in $X$ connecting $x_1$ and $x_2$, such that
$f\circ \gamma$ has finite length arbitrarily close to $\langle x_1 -x_2\rangle _f $.
Since $\ell_Y( f\circ \gamma ) \geq \ell _Y(g\circ \gamma)$ we deduce that 
$\langle x_1 -x_2\rangle _g$ is finite and not larger than $\langle x_1 -x_2\rangle _f$.

Assume on the other hand, that $g$ satisfies the conditions in the statement of the Lemma. In order to prove $f\unrhd g$  (rel. $Z$), consider an arbitrary curve
$\gamma$ in $X$. If $f\circ \gamma $ has infinite length, then 
$\ell _Y (f\circ \gamma ) \geq \ell _Y (g\circ \gamma)$.   

If $f\circ \gamma $ has finite length, then $\ell _Y (f\circ \gamma)= \ell _{\langle X \rangle _f} (\hat \pi _f \circ \gamma )$ by Lemma \ref{lem: zusatz}. By assumption,
the canonical map  $\mu:\langle X \rangle _f \to \langle X \rangle _g$ 
is $1$-Lipschitz and 
commutes with the projections $\hat \pi _f$ and $\hat \pi _g$. Thus,  $\hat \pi _g \circ \gamma$ is continuous and has  length at most  $\ell _{\langle X \rangle _f} (\hat \pi _f \circ \gamma )$.  Applying Lemma \ref{lem: zusatz} twice, we deduce $\ell _Y (f\circ \gamma ) \geq \ell _Y (g\circ \gamma)$.  
\end{proof}

The property of being  length-minimizing is inherited by restrictions:

\begin{Lemma} \label{lem:restrict}
Let $X$ be a topological space, let $Z, S \subset  X$ be closed. 
 Let  $Y$ be a metric space and 
  $f: X\rightarrow Y$ be  length-minimizing  relative to $Z$.  Then the restriction $f:S\to Y$ of $f$   is  length-minimizing  relative to
 $Z_S:=\partial S \cup (Z\cap S)$.
\end{Lemma}

\begin{proof}
Assume the contrary.
 Then there exists a map $h: S \rightarrow Y$ such that 
$f\big|_{S}\unrhd h$ (rel $Z_S$) together with a curve $\gamma_0 : I_0\rightarrow S$ such that $\ell_Y(f\circ\gamma_0 )> \ell_Y(h\circ\gamma_0 )$. 

Define  $g: X\rightarrow Y$ by setting  $g=h$ on $S$ and $g=f$ on $X\setminus S$.
The maps $h$ and $f$ agree on $\partial S$, hence $g$ is continuous. By construction, 
$g=f$ on $Z$.  Moreover, 
$$\ell_Y(h\circ \gamma _0)=\ell_Y(g\circ\gamma_0 ) < \ell_Y(f\circ\gamma_0 )\,.$$

Let now $\gamma:I\to X$ an arbitrary curve.  Then $f\circ \gamma$ and $g\circ \gamma$ agree
on the closed set $C= \gamma ^{-1} (\overline {X\setminus S})$.  The complement 
$I\setminus C$ is a countable union of open intervals $I_j$. For any $j$,  the length of the restriction
of $g\circ \gamma$ to $I_j$ does not exceed the  length of $f\circ \gamma \big |_{I_j}$,
since $f\big|_{S}\unrhd h$.  Computing the length via the $1$-dimensional Hausdorff measure $\mathcal H^1$, see \cite[Exercise 2.6.4]{burago2022course}, we deduce 
$$\ell(g\circ \gamma) \leq \ell (f\circ \gamma)\;.$$ 
This contradicts the length-minimality of $f$.
\end{proof}

As a consequence we obtain the following  \emph{non-bubbling} property:
\begin{Corollary} \label{cor:bubbling}
Let $X$ be a Peano space and let $f:X\to Y$ be length-minimizing relative to 
a closed subset $Z\subset X$.  Then, for any $y\in Y$ and  any connected component $U$ of $X\setminus f^{-1} (y)$ the intersection $ U\cap Z$ is not empty.
\end{Corollary}

\begin{proof}
Assume $U\cap Z= \emptyset$ and let $S$ be the closure of $U$ in $X$.
 Since connected components are always closed in the ambient space,
the boundary $\partial S$ of   $S$ in $X$, satisfies 
$\partial S \subset     f^{-1} (y)$. Moreover, $S\cap Z  = \partial S \cap  Z$.

By Lemma \ref{lem:restrict},  $f:S \to Y$ is length-minimizing relative to $\partial S$.  The constant map $g:S \to Y$, which sends every point to $y$ satisfies $f\unrhd g$.  The  length-minimality  of $f$ implies that the image of every curve in $U$ has length $0$. Therefore,  $f$ is constant on $U$, hence on $S$.
Thus, 
$U$  is contained in $f^{-1} (y)$, which is impossible.
\end{proof}

\begin{Lemma} \label{cor: contained}
Let $X$ be a topological space, $Z\subset X$ be closed.  
Let $Y$ be  a \textsc{CAT}$(\kappa )$ space. Let  $K\subset Y$ be  closed, convex
and   contained  in a ball of radius $<\frac {R_{\kappa}} 2$ in $Y$.

Let $f:X\to Y$ be a length-connected,  length-minimizing map relative to $Z$.
 If $f(Z)$  is contained in $K$ then $f(X)$ is contained in $K$.
%
\end{Lemma}

\begin{proof}
Assume that there exists a point $x\in X$ with $f(x) \notin K$.  We find a curve 
$\gamma _0$ in $X$ between a point $z\in Z$ and $x$ such that $f\circ \gamma _0$ has finite length.

Due to \cite[Theorem 1.1]{lytchak2021short}, there exists a $1$-Lipschitz retraction 
$\Psi: Y\to K$, which decreases the length of any rectifiable curve of positive length 
not completely contained in $K$.   Hence, the composition $g:=\Psi \circ f$ satisfies 
$f\unrhd g$ (rel. $Z$) and  $\ell_Y(g\circ \gamma _0) < \ell_Y(f\circ \gamma _0)$.
This contradicts the  length-minimality of $f$.
\end{proof}

\begin{Lemma} \label{lem: coincide}
Under the assumptions of Lemma \ref{cor: contained}, let 
 $g:X\to Y$ satisfy $f\unrhd g$ (rel. Z). Then 
$f=g$. 
\end{Lemma}

\begin{proof}
By definition, $g$ is  length-minimizing relative to $Z$ as well. By Corollary \ref{cor: contained}, the images of $f$ and of $g$ are contained in  $K$.

Denote by $\Delta \subset K\times K$ the diagonal.  We apply \cite[Corollary 1.2]{lytchak2021short} and obtain a  $1$-Lipschitz retraction $\Pi :K\times K \to \Delta$, such that
  $\Pi$ decreases  the length of any curve of finite positive length not  completely contained in $\Delta$.

We identify $\Delta$ with $K$, rescaled by $\sqrt 2$ and 
 consider the map $h:X\to Y$  
$$h(x):= \Pi (f(x),g(x))\;.$$ 
Then $h$ coincides with $f$ and with $g$ on $Z$. For any curve $\gamma$ in $X$, we have
$$\ell_Y(h\circ \gamma) \leq \ell_Y(f\circ \gamma)=\ell_Y(g\circ \gamma).$$
Whenever $\ell_Y(f\circ \gamma)$ is finite and positive,  and  $f\circ \gamma $ does not coincide with $g\circ \gamma$, the above inequality is strict.

If $f$ does not coincide with $g$ on $X$, we find a curve $\gamma_0$ in $X$, such that $f\circ \gamma_0$ and $g\circ \gamma_0$ are different curves of finite length, since  $f$ is length-connected. 
We infer  $f\unrhd h$ (rel. $Z$) and that $f$ is not  metric minimizing. Contradiction.
\end{proof}

We will apply  Lemmas \ref{lem: coincide}, \ref{cor: contained} only  for  balls $K= B_r(y)\subset Y$ with  $r<\frac {R_{\kappa}} 2$.  However,  these lemmas
 might be useful in other cases as well, cf. \cite{PS-2}.

\subsection{Length-minimizing graphs}

The results in this Subsection are contained explicitly or between the lines
\cite[Section 5]{petrunin2019metric}.

By a \emph{finite graph} we  understand a connected, finite $1$-dimensional CW-complex.
Thus, we allow multiple edges and edges connecting a vertex with itself.

\begin{Proposition} \label{prop: graph}
Let $Y$ be a geodesic metric space, $\Gamma$ a finite  graph and $A$ a set of its vertices.   
 If $f:\Gamma \to Y$ is  
 length-minimizing  relative to   $A$ then the restriction of $f$ to any edge is a geodesic. 
\end{Proposition}

\begin{proof}
By Lemma  \ref{lem:restrict},  the restriction $f_0:E\to Y$ of $f$ to any edge  $E\subset \Gamma$
is  length-minimizing  relative to its endpoints. We can identify $E$  with the interval $[0,a]$, such that $a$ is the distance 
$a=\langle f(0)-f(a) \rangle _Y$ between the endpoints. 
 Consider a geodesic $\gamma:[0,a]\to Y$ with the same endpoints as $f$, parametrized by arclength.  Then $C:=\gamma ([0,a])$ is an isometric embedding of the interval $[0,a]$. 
Since intervals are injective metric spaces, cf. \cite[Lecture 3]{petruninmetric},  there exists a $1$-Lipschitz retraction $\Phi:Y\to C$.   Hence, for 
$g:=\Phi \circ f_0$ we have $f   \unrhd g$ (rel. $\{0,a\}$).  

Hence, also the map $g:[0,a]\to C$ is  metric minimizing relative to  $\{0,a\}$. Due to the non-bubbling result, Corollary \ref{cor:bubbling}, the map  $g$ provides a monotone parametrization of the simple arc $C$.  Hence, the length of the curve $g$ is exactly the length of the geodesic $\gamma$.  Since $f_0$ is length-minimizing, the length 
of $f_0$ cannot be larger than the length of $g$.  Hence, $f_0$ is  a geodesic. 
\end{proof}

It follows that any length-minimizing graph $f:\Gamma \to Y$ as in Proposition \ref{prop: graph} is length-continuous.  Any connected subset of $\Gamma$ 
  is arcwise connected.
Therefore, the canonical map 
 $\tau_f:\langle \Gamma \rangle _f \to |\Gamma |_f$ is a homeomorphism.  Moreover,
the  space $\langle \Gamma \rangle _f$ is a finite union of geodesics intersecting only at their endpoints. Thus, $\langle \Gamma \rangle _f$ is 
finite topological graph with a geodesic metric.
A vertex in $\langle \Gamma \rangle _f$ corresponds  to  connected  subgraphs of $\Gamma$ sent by $f$ to a single point. Any edge $\hat E$ of $\langle \Gamma \rangle _f$ is the  image of an edge $E$  in $\Gamma$ and the restricton $\hat \pi _f: E\to \hat E$ is monotone. 
 The induced map $\hat f:\langle \Gamma \rangle _f \to Y$ sends any edge isometric onto a geodesic in $Y$.  Finally, any curve in $\langle \Gamma \rangle _f $ lifts to a curve in $\Gamma$.  By definition, the last statement implies:

\begin{Corollary} \label{cor: graph}
Let $f:\Gamma \to Y$ be length-minimizing relative to $A$, as in Proposition \ref{prop: graph}.  Then $\hat f: \langle \Gamma \rangle _f \to Y$ is length-minimizing relative to $\hat \pi _f (A)$.
\end{Corollary}

The following result is proved in detail in   \cite[Proposition 5.2]{petrunin2019metric}
for $\kappa =0$. Since the proof only requires the first variational inequality for distances \cite[Inequality 6.7]{alexander2014alexandrov} and Reshetnyak's majorization theorem \cite[Theorem 9.56]{alexander2014alexandrov} in spaces of directions, the proof applies literally in  \textnormal{CAT}$(\kappa )$ spaces:

\begin{Lemma}\label{lem: winkel}
Let $Y$ be  $\textsc{CAT}(\kappa)$.
Let $\Gamma$ be a finite geodesic graph and $A$ be a subset of its vertices. Let $f: \Gamma\rightarrow Y$  be  length-minimizing relative to $A$ and assume that 
the restriction of $f$ to any edge is an isometry.

Let $p\in \Gamma \setminus A$ be a vertex. Let $\gamma _1,..., \gamma _n$ be the images in $Y$ of the edges in $\Gamma$  starting  in $p$ and enumerated in an arbitrary order.
Then the sum of the $n$ consecutive angles  satisfies
\begin{center}
$\measuredangle _p(\gamma _1, \gamma _2)+\ldots +\measuredangle_p
(\gamma_{n-1}, \gamma _n)+\measuredangle _p (\gamma _n,\gamma _1) \geqslant 2\pi$.
\end{center}
\end{Lemma}

In particular,  any vertex $p\in \Gamma \setminus A$ is contained in at least $2$ edges 
of $\Gamma$.

\subsection{Length-minimizing discs}
A \emph{length-minimizing disc} will denote  a 
map $f:\mathbb D \to Y$ length-minimizing relative to $\mathbb S^1$.

Due to Corollary  \ref{cor:bubbling} and Lemma \ref{lem: banah}, for any length-connected and length-minimizing disc $f:\mathbb D \to Y$, the spaces
$ | \mathbb D |_f $ is a disc retract. If, in addition, $f$ is length-continuous, then
 the map $\tau _f :\langle \mathbb  D \rangle _f\to | \mathbb D |_f $ is a homeomorphism preserving the length of all curves.




\begin{Lemma} \label{cor: items}
 Let
$f:\mathbb D\to Y$ be length-connected and  length-minimizing disc in a $\textsc{CAT}(\kappa)$  space $Y$. 
 Let $G$ be a Jordan curve of length $l <  2 \cdot {R_{\kappa}} $ in the disc retract $|\mathbb D|_f$.
Denote by $J$ the  closed disc bounded by $G$ in  $|\mathbb D|_f$.
Then 
\begin{enumerate}
\item  The restriction $\bar f:J\to Y$ is a length-connected and  length-minimizing disc. 
\item The image $\bar f(J)$ is contained in a ball of radius $r <\frac {R_{\kappa}} {2} $  in $Y$.
\item If $f$ is length-continuous then so is $\bar f:J\to Y$.
\end{enumerate}
\end{Lemma}

\begin{proof}
The map $\bar f$ is $1$-Lipschitz, hence   length  of the closed curve $\bar f (G)$ is less than $ 2 \cdot {R_{\kappa}} $. Reshetnyak's majorization theorem \cite[Theorem 9.56]{alexander2014alexandrov} implies that $\bar f(G)$ is contained in 
  a ball  $B$ of radius $r <\frac {R_{\kappa}} {2} $  in $Y$.

Denote by $\tilde G$ and $\tilde J$ the preimages of $G$ and $J$ in $\mathbb D$, respectively.  

For any point $p\in J$ consider any preimage $\tilde p\in \tilde J$ of $p$.  We find a curve $\gamma$  in $\mathbb D$, which connects $\tilde p$ with a point  in $\tilde G$ and such that $f\circ \gamma$ has finite length. By cutting the curve, if needed, we may assume that
$\gamma$ is contained in $\tilde J$. Then the projection $\bar \gamma =\bar \pi _f\circ \gamma$ is a curve  connecting $p$ to a point on $G$, such that $\bar f\circ \bar \gamma$ has finite length.      

Since $G$ has finite length and $\bar f$ preserves all lengths, 
$\langle p-q \rangle _{\bar f}$ is finite, for  any $p\in J$ and any $q\in G$ . By the triangle inequality,  $\bar f$ is length-connected.

Assume now that $g:J\to Y$   satisfies $\bar f\unrhd g$ (rel.  $G$).   Set 
$$\tilde g:=g\circ \bar \pi _f :\tilde J\to Y\,.$$ 
Then $\tilde g$ coincides with $f$ on 
$\tilde G$ and  $\tilde J\setminus \tilde G$ does not intersect the boundary $\mathbb S^1$.   
Due to Lemma \ref{lem:restrict}, the restriction $f:\tilde J\to Y$ is length-minimizing relative to $\tilde G$. 

The assumption $\bar f\unrhd g$ (rel.  $G$) implies $f|_{\tilde J} \unrhd \tilde g$.
Due to Corollary \ref{cor: contained}, $ f(\tilde J)=\bar f(J)$ is contained in the ball $B$, proving (2).

Due to Lemma \ref{lem: coincide},  $f$ equals $\tilde g$ on $\tilde J$.
Hence, $\bar f= g$.  This proves (1).

Let, finally, $f$ be length-continuous. By Lemma \ref{lem: banah},
 $\langle X \rangle _f$ 
is a geodesic metric space  and $\tau _f : \langle X \rangle _f \to |X |_f$ is  a homeomorphism.  Consider the disc  $\hat J=\tau _f ^{-1}(J) \subset \langle X \rangle _f$. 
Its boundary $\hat G= \tau _f^{-1} (G) $ has finite length, therefore,
the length metric  of the subset $\hat J$ of  $ \langle X \rangle _f$ induces the subset topology of $\hat J$, cf.  \cite[Lemma 2.1]{LW-param}.  The restriction $\hat f :\hat J\to Y$ is $1$-Lipschitz, if $\hat J$ is equipped with its length metric.  Hence, $\hat f:\hat J\to Y$, and therefore also $\bar f:J\to Y$, is length-continuous. 
 \end{proof}

\section{An existence result }  \label{sec: exist}
The following result is essentially  \cite[Proposition 5.1]{petrunin2019metric}.

\begin{Lemma} \label{lem: existence}
Let $X$ be a Peano continuum, $Z$ a closed subset of $X$ and $Y$ a \textsc{CAT}$(\kappa )$  space. Let the map $f:X\to Y$ be length-continuous.   If $f(Z)$ is a contained in a closed ball 
$ B \subset Y$ of radius $\leq \frac{R_{\kappa}} 2$ then
there exists a  map $g:X\to Y$, length-minimizing  relative to $Z$, with $f\unrhd g$ (rel. $Z$).
\end{Lemma}

\begin{proof}
We consider the set  $\mathcal F$ of all maps $h:X\to Y$ which coincide with $f$ on $Z$
and satisfy $f\unrhd h$ (rel. $Z$).    We are looking for a minimal element in the set $\mathcal F$ with  the partial ordering 
 given by $\unrhd$ (rel. $Z$).

By  \cite[Theorem 1.1]{lytchak2021short}  there exists a  $1$-Lipschitz retraction $\Phi:Y\to B$. Hence, for any $h\in \mathcal F$, we have $\Phi\circ h \in \mathcal F$ and 
$h\unrhd \Phi \circ h$.

We replace $\mathcal F$ by its subset $\mathcal F_0$ of all maps $h\in \mathcal F$ whose image is contained in the ball $B$.  The existence of the retraction $\Phi$ implies that an element  $g$ in $\mathcal F_0$ minimal with respect to the partial ordering $\unrhd $ (rel. $Z$)  will  be  length-minimizing.

  In order to find such a minimal  element  in $\mathcal F_0$ we apply Zorn's lemma.  It thus suffices to show that for any totally ordered subset $\mathcal T$ of $\mathcal F_0$ there exists 
an element $g\in \mathcal F_0$ with $h\unrhd g$ (rel. $Z$) for all $h\in \mathcal T$.

For any $h\in \mathcal T$, the identity $id:Y\to Y$  induces a $1$-Lipschitz map
$\mu:\hat Y _f \to \hat Y_h$.  The map $h$ then factorizes as  $h= \hat h \circ  \mu \circ \hat \pi _f$. 
Since the
maps $\hat h $ and $\mu$ are $1$-Lipschitz and $\hat \pi _f$ is continuous by assumption,  the family  $ \mathcal T$ is   equicontinuous.

Fix  a dense sequence $(x_n) $ in the compact space $X$.
For any pair $(n,m)$,   set 
$$d_{m,n} :=\inf \{ \langle x_n-x_m \rangle _h \; \; | \; \; h\in \mathcal T\}\;.$$
For any triple of natural numbers $(n,m,k)$ we find some $h=h_{k,n,m}$ in $\mathcal T$
such that 
$$d_{m,n}+ \frac 1 k > \langle x_n-x_m \rangle _h\;.$$
We enumerate the set of all triples and use a diagonal sequence argument in order to find
a sequence $   h_1  \unrhd h_2 \unrhd  h_3 \unrhd ...$ of elements in $\mathcal T$ such that for any $x_n,x_m$ 
$$\lim _{i\to \infty} \langle x_n-x_m \rangle _{h _i}  =d_{m,n}\;.$$

If $Y$ is proper, hence $B$ compact,  we take a  convergent subsequence of the equi-continuous sequence $h_i$
and obtain a uniform limit $g:=\lim _i h_i:X\to B.$

If $Y$ is non-proper, we take the ultralimit
$$h_{\omega}=\lim_{\omega} h_i: X\to B^{\omega }$$
from $Y$ to the ultracompletion $B^{\omega}$ of $B$.
Then we apply \cite[Theorem 1.1]{lytchak2021short} again and find a $1$-Lipschitz retraction $\Psi: B^{\omega} \to B$. Set now 
 $$g:=\Psi \circ h_{\omega}.$$

Since all $h_i$ agree with $f$ on $Z$, so does $g$.
Since length of curves is semi-continuous under Gromov--Hausdorff convergence and under ultraconvergence,  we have
$$\lim \ell_Y(h_i \circ \gamma) \geq \ell_{B^\omega } (h_{\infty } \circ \gamma) \geq \ell _B(g\circ \gamma) \;,$$
for any curve $\gamma$ in $X$.
In particular, $g$ is contained in $\mathcal F_0$. Moreover, for any  $n,m,i$
$$d_{m,n}\leq  \langle x_n-x_m \rangle _g  \leq \langle x_n, x_n \rangle _{h_i}  \;.$$
Thus, equality holds  on the left.  By the assumption  on  $h_i$,  for any $h\in \mathcal T$,
$$\langle x_n -x_m \rangle _h \geq \langle x_n -x_m \rangle _g\;$$
Since the sequence $x_m$ is dense in $X$ and, for all $h\in \mathcal T$, the distance
$\langle \ast -\ast \rangle _h$ is continuous on $X$, we deduce
$$\langle x -z \rangle _h \geq \langle x -z \rangle _g\;,$$
 for any pair of points
$x,z\in X$ and any $h\in \mathcal T$.

Thus, $h\unrhd g$ (rel. $Z$), for all $h\in \mathcal T$.   This finishes the proof. 
\end{proof}

The CAT$(\kappa$)-assumption and the assumption that the image is contained in a small
ball is only used for the existence of a $1$-Lipschitz retraction $\Pi:Y^{\omega} \to Y$.
Thus, in the proper setting the proof provides the following result, which we state for the sake of completeness: 

\begin{Proposition}
Let $X$ be a Peano space, $Z\subset X$ closed. Let $f:X\to Y$ be a length-continuous  map into a proper metric space $Y$. Then there exists a length-continuous   map $g:X\to Y$, which is  length-minimizing relative to $Z$ and satisfies $f\unrhd g$ (rel $Z$).   
\end{Proposition}

The first part of the proof of Proposition \ref{prop: graph}, shows that  for any finite graph $\Gamma$ and any continuous map $f:\Gamma \to Y$ into a geodesic space $Y$ there exists a
map $g:\Gamma \to Y$ such that $f\unrhd g$ and such that the restriction of $g$ 
to any edge is a geodesic in $Y$.  Since any such $g$ is length-continuous, we deduce from 
Proposition \ref{lem: existence}
\begin{Corollary} \label{cor: graphex}
Let $\Gamma $ be a finite graph and $A\subset \Gamma$ a set of vertices.  Let   $B$ be a ball of radius $\leq \frac {R_{\kappa}} 2$
in a  \textsc{CAT}$(\kappa )$  space $Y$. For any continuous map $f:\Gamma \to B$,
  there exists a length-minimizing map $h:\Gamma \to B$ relative to $A$ such that $f\unrhd h$.
\end{Corollary}

\section{Main result}
\label{KL}

The following proposition is the technical heart of all results in this paper. This is a slight generalization of \cite[Key-Lemma 6.2]{petrunin2019metric}. The proof  follows
  \cite{petrunin2019metric} filling  some details and some additional arguments in positive curvature.

\begin{Proposition} \label{KeyLemma}
Let $\Gamma$ be a  finite graph, which is the $1$-skeleton of a triangulation of the disc $\mathbb D$. Let $\mathcal V$ be the set  of vertices of $\Gamma$  and $A:=\mathcal V\cap \mathbb S^1$. Let   $B$ be a ball of radius $<\frac {R_{\kappa}} 2$ in some  $\textsc{CAT}(\kappa)$ space $Y$. Let $f:\Gamma \to B$ be continuous. Let $0<\varepsilon < \frac {R_{\kappa}} 2$ be such that,
for the vertices $x_1,x_2,x_3$ of  any triangle $T$ in the triangulation of $\mathbb D$, the distances satisfy $\langle f(x_i)- f(x_j) \rangle _Y \leq \varepsilon$.

Then there exists  some $\textsc{CAT}(\kappa)$ disc retract  $W$, a $1$-Lipschitz map $q:W\to Y$
and a map $p:\mathbb D\mathbb \to W$ with  the following properties
\begin{enumerate}

\item  All fibers of $p$ are contractible and $p(\mathcal V)$ is $\varepsilon$-dense in $W$.

 \item $q\circ p|_{A} =f|_A$.

\item For every curve $\gamma$ in $\Gamma$, we have $\ell _Y (f\circ \gamma)\geq \ell _W (p\circ \gamma)  = \ell _Y (q\circ p \circ \gamma)$.
\end{enumerate}

\end{Proposition}

\begin{proof} 
 We may replace $Y$ by $B$ and  assume $Y=B$.
Due to Corollary \ref{cor: graphex}, we find a map $h:\Gamma \to Y$ length-minimizing relative to $A$ such that $f\unrhd h$ (rel. $A$).

 Consider the induced length-metric space $\langle \Gamma \rangle _h$, which is a geodesic graph.    Due to Corollary \ref{cor: graph}, the map $\hat h:\langle \Gamma \rangle _h \to B$ is  length-minimizing relative to $\hat \pi _h (A)$. By Proposition \ref{prop: graph},  the restriction 
of $\hat h$ to any edge of $\langle \Gamma \rangle _h$ is an isometry and
the projection 
$\hat \pi  _h:  \Gamma \to  \langle \Gamma  \rangle _h$ is  monotone
 Thus, the conclusions of Lemma \ref{lem: winkel} are valid for the map $\hat h:\langle \Gamma \rangle_h \to Y$.

For any triangle $T \subset \mathbb D$  bounded by 3 edges of $\Gamma$ with vertices $x_1,x_2,x_3$  we consider the unique, possibly degenerated  triangle $T^{\kappa}$  in the surface  $M^2_{\kappa}$ of constant curvature $\kappa$, such that the sides of $T^{\kappa}$ have distances equal to  
 $\langle x_i-x_j \rangle _{\langle \Gamma _h \rangle} = \langle  h(x_i)-h(x_j) \rangle _Y
 \leq \varepsilon$.   
We glue all these  triangles  $T^{\kappa}$ (possibly degenerated to a vertex or to an edge) together in the same  way as the corresponding triangles $T^{\kappa}$  are glued in $\mathbb D$.
We denote the obtained metric space by $W$. 

The metric  space $W$ is glued from triangles of curvature $\kappa$
and, possibly, some edges, cf. \cite[definition 12.1]{alexander2014alexandrov}.
 The gluing maps are simplicial. A priori, edges or vertices of a single triangle $T$ may be identified in $W$. 

We have a tautological length-preserving  map $\iota : \langle \Gamma \rangle _h \to W$. 
Some caution  is  in order: 
two non-degenerate edges in a triangle whose third side degenerates to a point are sent by $\iota$ to the 
same curve in $W$, hence $\iota$ may be non-injective. However, all maps defined below respect the identifications on $\langle \Gamma \rangle _h$ given by $\iota$. Thus, we will
identify $\langle \Gamma \rangle _h$ with $\iota (\langle \Gamma \rangle _h)$ below.

  For any triangle $T \subset \mathbb D$, the restriction of 
$\hat \pi _h:\Gamma \to \langle \Gamma \rangle _h$  to $\partial T$ defines a map
$f_T:\partial T\to \partial T^{\kappa}$.  We claim that $f_T$ admits an extension to a map $\tilde f_T:T\to T^{\kappa}$, such that any fiber  of $\tilde f_T$ is topologically  a point, an interval or a disc.

If $T^{\kappa}$ degenerates to a point $w\in W$ we choose the constant map $\tilde f_T\to \{w\}$.

If $T^{\kappa}$ degenerates to an edge $E$, the map $f_T$ sends one side  of $\partial T$ 
to a point and the two other edges monotonically to $E$.  Identifying $T$ with a Euclidean triangle, we then obtain a unique extension $\tilde f_T :T\to E$ of $f_T$  with all fibers convex subsets.
 This finishes the construction  in the degenerated case.

If  none of the three sides of $T^{\kappa}$ degenerates  to a point, $T^{\kappa}$  is a topological disc with boundary $\partial T^{\kappa}$.
Identifying the two triangles $T$ and $T^{\kappa}$ with the disc and coning the map $f_T$, we obtain an extension map $\tilde f_T :T\to T^{\kappa}$  of $f_T$. The preimages of  points under the map $\tilde f_T$ are 
 (possibly degenerated) compact intervals.

All the maps $\tilde f_T$ constructed above  coincide on common edges and points of different intersecting triangles $T$.  Thus, all $\tilde f_T$ glue together to  a continuous map $p:\mathbb D\to W$. By construction,  the preimage $P:=p^{-1} (w)$  of  any point  $w\in W$  is a Peano continuum and the intersection of $P$ with any triangle $T$ is contractible.

We claim that any such  preimage  $P=p^{-1} (w)$ is contractible. 
Otherwise, we would find a  non-contractible circle $S$ in $P$. Denote by $O$ the open Jordan domain of $S$ in $\mathbb D$.
 Since the intersection of $P$ 
with each triangle $T$ is empty or contractible, $S$ is not contained in any triangle. Hence, $O$ has a non-empty intersection with $\Gamma$. The open disc  $O$ does not contain 
points in $A$. By Corollary \ref{cor:bubbling}, applied to the length-minimizing map $h$, the  subset $O\cap \Gamma$ must be contained in $h^{-1} (w)$, hence in $P$.  Since any triangle intersects $P$ in an empty or in  a contractible set, we deduce $T\cap O\subset P$, for all triangles $T$, with $T\cap O\neq \emptyset$.  Therefore, $O\subset P$. This  contradicts 
 the non-contractibility of $S$.

Hence,  $p$ has contractible fibers. Therefore, $W$ is a disc-retract \cite[Lemma 3.3]{petrunin2019metric}.  In particular, $W$ is contractible.

The diameter of any triangle $T^{\kappa}$ equals to the maximum of the distances between its vertices. Hence, this diameter is at most $\varepsilon$, by assumption. Therefore, the vertices of $T^{\kappa}$ are $\varepsilon$-dense in $T^{\kappa}$ and $p(\mathcal V)$ is $\varepsilon$-dense in $W$.

For any edge $E$ in $\langle \Gamma \rangle _h$, the map $\hat h: E\to Y$ is an isometric embedding.   For any degenerated triangle $T^{\kappa}\subset W$, this gives an isometric embedding  $\tilde h_T:=\hat h:T^{\kappa} \to Y$.   

If $T^{\kappa}$ is non-degenerated,  then $T^{\kappa}$ is the \emph{comparison  triangle}
of the triangle $\hat h (T ^{\kappa})$ in $Y$.  Reshetnyak's majorization theorem, \cite[Theorem 9.56, Proposition 9.54]{alexander2014alexandrov} implies that $\hat h$ extends to a $1$-Lipschitz map $\tilde h_T:T^{\kappa} \to Y$.

 The maps $\tilde h_T$ coincide on common edges and vertices, where they are given by $\hat h$. Hence, 
these maps glue together to  a  $1$-Lipschitz map $q:W\to Y$,  whose 
restriction to $\langle \Gamma \rangle _h \subset W$
is $\hat h$.

  By construction,   $q\circ p|_A=h|_A=f|_A$.
For every curve $\gamma $ in $\Gamma$, we have
$$\ell _Y(f\circ \gamma) \geq \ell _Y (h\circ \gamma) = \ell _W (p\circ \gamma) \,.$$

It remains to show,  that the space $W$ is CAT$(\kappa )$.   By construction, 
$W$ is obtained by gluing together (simplically) some number of triangles   of constant curvature $\kappa$.   Since the space $W$ is a disc retract, the space of directions $\Sigma _w W$  at any point  $w\in W$ topologically embeds into a circle. Hence, $\Sigma _w W$
 is topologically  either a circle  or a disjoint union of (possibly degenerated) intervals.

If $W$ is not locally CAT$(\kappa )$, we find a point $w\in W$ and a closed local geodesic of length less than $2\pi$ in $\Sigma _wW$ \cite[Theorem 12.2]{alexander2014alexandrov}.
Then  the whole space $\Sigma _wW$ must  be  a circle of length less than $2\pi$.  Hence, $w$ must be a vertex of $\langle \Gamma \rangle _h$ and the cyclic sum of angles   $\alpha _j$  (in triangles $T_j ^{\kappa}$) adjacent to $w$  is less than $2\pi$.

By construction and angle comparison, the angles in $T^j_{\kappa}$ are not smaller than the corresponding angles in  $\hat h (\partial T_i ^{\kappa} )\subset Y$.

Thus, the cyclic sum of all angles $\angle _{\hat h (w)} (\hat h (e_i),\hat h (e_{i+1}))$ in $Y$ is less than $2\pi$, where $e_i$ run over all edges in $\langle \Gamma \rangle _h$ adjacent  to $w$ in $\langle \Gamma \rangle  _h$.  This contradicts the length-minimality of $\hat h$ and  Lemma \ref{lem: winkel}.

Therefore, the space $W$ is locally CAT($\kappa$). Since $W$ is simply connected,  we may
apply  the Cartan--Hadamard theorem, \cite[Theorem 9.6]{alexander2014alexandrov}, and deduce that   $W$ is CAT$(\kappa )$, if $\kappa <0$.
 The case $\kappa >0$ requires an additional argument.

Thus, we assume now $\kappa >0$.

The fact that the  total angle at any interior point is at least $2\pi$ implies that no two sides of a non-degenerate triangle $T^{\kappa}$ are identified in $W$.  Therefore, the triangles $T^{\kappa}$ define a triangulation of $W$ with $1$-skeleton $\iota (\langle \Gamma \rangle _h)$.

  If $W$ is not CAT($\kappa $), we find 
an isometrically embedded circle $S$ of length $<2R_{\kappa}$ in $W$ \cite[Theorem 2.2.11]{bowditch1995notes}. 
 Moreover, we find  and fix such a circle with the smallest possible length. Upon rescaling we may assume that $S$ has length $2\pi$. Then $\kappa <1$. Since there are no closed geodesics of length $<2\pi$, the space $W$ is CAT(1) and locally CAT($\kappa$)

Consider the subdisc $W_0$ of $W$ bounded by $S$, which
is a convex subset of $W$.  We are going to construct a $1$-Lipschitz map $\tilde q:W_0 \to Y$, which equals  $q$ on $S$ and such that $\tilde q$ shortens the length of some curve $\gamma \subset \langle \Gamma \rangle _h \cap W_0$.  Once $\tilde q$ is constructed, the map $$j : \langle \Gamma \rangle _h \to Y $$  given by $q$ on $\langle \Gamma \rangle _h \setminus W_0$ and 
by $\tilde q$ on $W_0$ would satisfy $\hat h \unrhd j$ (rel. $A$) and contradict the length-minimality of $\hat h$.

It remains to construct the $1$-Lipschitz $\tilde q:W_0 \to Y$, which shortens  the length of some curve in 
$\langle \Gamma \rangle _h \cap W_0$.

Denote by $H$ the hemisphere of curvature $1$ with pole $o$.
Glue $H$ to $W_0$ identifying $S$ with the boundary of the hemisphere $H$.  The arising space $$W_1:=W_0\cup _S  H$$ is CAT$(1 )$ by Reshetnyak's gluing theorem.  There is a canonical 
$1$-Lipschtitz retraction $\Pi:W_1\to H$, defined as follows (compare \cite[Lemma 1.3]{Jost} \cite{lytchak2021short}):  
\begin{itemize}
\item  $\Pi (x)=x$ if $|o-x|_{W_1} \leq \frac \pi 2$.
\item $\Pi (x)=o$ is $|o-x|_{W_1} \geq \pi $.
\item $\Pi (x)=\eta_x (\pi- |o-x|_{W_1})$ if $\pi > |o-x|_{W_1} > \frac \pi 2$, where 
$\eta_x$ is the geodesic  from $o$ to $x$ in $W_1$ parametrized by arclength. 
\end{itemize}

The hemisphere  $H$ has constant curvature $1 < \kappa$. We apply the Kirzsbraun--Lang--Schroeder extension theorem, \cite{Lang}, \cite[Theorem 10.14]{alexander2014alexandrov}, and find
a  $1$-Lipschitz map $g:H\to Y$ which extends the $1$-Lipschitz map $q:S\to Y$.

We define the $1$-Lipschitz map $\tilde q:W_0\to Y$  as the composition
$\tilde q= g\circ \Pi$.  Then $\tilde q$ coincides with $q$ on $S$.
 It remains to find a curve in $G:=\langle \Gamma \rangle _h \subset W_0$,
whose length is strictly contracted by $\tilde q$. 

 Assume on the contrary, that $\tilde q$
preserves the length of all curves in $G$. 
The closed geodesic  $S$ cannot be contained in a single triangle $T^{\kappa}$ nor can 
it  intersect twice the same side of the same  triangle $T^{\kappa}$. It follows that 
$G$ together with $S$ constitute the $1$-skeleton of a triangulation of $W_0$.

By assumption, $\tilde q$ is $1$-Lipschitz on $G$. On the other hand $\hat h$  is length-preserving and length-minimizing on $G$ relative to $G\cap S$.
 Therefore, $\tilde q$ is length-minimizing on $G$ relative 
to $G\cap S$ as well.  Hence, $\tilde q$ is an isometry on every edge of the graph $G$.
Since $g$ is $1$-Lipschitz, we deduce that the $1$-Lipschitz map $\Pi :W_0 \to H$ 
restricts to an isometry to every edge of $G$.  

We find a triangle $T_0 =xpz$ in the triangulation of $W_0$ defined by $G\cup S$, such that
one side  $xz$ of $T_0$ lies on $S$. The definition of the map $\Pi$ and the equality $$|p-x|_{W_0} = |p-x|_{W_1} = |\Pi (p)- \Pi (x)| _H$$
imply that  $|p-x|_{W_0} \leq \frac {\pi} {2}$  and that 
the triangle $opx$ is  of constant curvature $1$.   Since $W_0$ is locally CAT($\kappa$), it implies that the triangle $opx$ is degenerated. Hence $px$ meets $S$ at $x$ orthogonally.
The same is true for the geodesic $pz$.

Since $W_0$ is CAT($1$), we deduce that $p$ has distance $\frac {\pi} 2$ to $x$ and to $z$ and that the triangle $pxz$ has constant curvature $1$. Since $W_0$ is locally CAT($\kappa$), this implies that the triangle $T_0$ is degenerated, which is impossible.

This contradiction shows that $W$ is CAT($\kappa$) and finishes the proof.
\end{proof}

We are now in position to provide

\begin{proof}[Proof of Theorem \ref{Theorem:main}]
Thus,  let $Y$ be CAT$(\kappa)$ and $f:\mathbb D \to Y$ be  length-continuous and length-minimizing relative to $\mathbb S^1$.
  We need to prove that $\langle \mathbb D\rangle _f$ is CAT$(\kappa )$.

Due to Lemma \ref{lem: banah},  $\langle \mathbb D\rangle _f $  is a disc retract.

It is sufficient to prove, that any Jordan triangle  $G$ (thus a Jordan curve, built by 3 geodesics) of length $<2R_{\kappa}$ is not thicker than its comparison triangle in $M^2_{\kappa}$.   We fix $G$ and denote by  $J$  the closed disc bounded by $G$ in the disc retract $\langle \mathbb D\rangle _f$.  

Then the restriction $\hat f:J\to Y$ is length-minimizing relative to $G$ and its image is contained in
a ball $B\subset Y$ of radius $r<\frac {R_{\kappa}} 2$, by Corollary \ref{cor: items}. 
Moreover, $\hat f :J\to Y$ preserves the length of all curves in $\langle J\rangle _{\hat f}$.

Since the canonical  embedding $\langle J \rangle _{\hat f} \to \langle \mathbb D \rangle _f$  is $1$-Lipschitz, it suffices to prove that $G$, considered as a subset of  $\langle J\rangle _{\hat f}$, is not thicker than its comparison triangle.  Therefore, it is sufficient to prove that $\langle J \rangle _{\hat f}$ is CAT$(\kappa)$.

Thus, we may replace $\mathbb D$ by $J$, $f$ by $\hat f$ and $Y$ by $B$ reducing our task to the following situation: 

We have a geodesic metric space $J$ homeomorphic to $\mathbb D$, whose boundary $G$
is a geodesic triangle of perimeter less than $2R_{\kappa}$. We have  a  CAT$(\kappa)$ space $Y$ which is contained in a ball of radius less than   $\frac {R_{\kappa}} 2$ around some of its points.  We have a $1$-Lipschitz map $f:J\to Y$ which is a light map. The map $f$ preserves the length of any curve and is length-minimizing relative to $G$.  Under these assumptions we need to verify that $J$ is CAT($\kappa$).

The above will be the standing assumption until the end of the proof. 

 For any natural $n$ we find a finite, connected,  piecewise geodesic graph $\tilde  \Gamma_n$,
which contains $G$ and satisfies the following  conditons, see \cite[Theorem 1.2]{CR} and
\cite[Proposition 5.2]{NR}: 

 The boundary of any component $T_0$ of $J\setminus \tilde \Gamma _n$ is a Jordan triangle and the closure $T$ of $T_0$ has diameter less than $\frac 1 n$.  The embedding of $\tilde \Gamma _n$ with its intrinsic metric into $J$ is a $\frac 1 n$-isometry, thus,
$$|x-z|_{\langle \tilde \Gamma _n  \rangle} \leq |x-z|_J  +\frac 1 n \;,$$
for all $x,z\in \tilde \Gamma _n$.  

Note that  the graph $\tilde \Gamma _n$ provided by \cite[Proposition 5.2]{NR} does not need to be the $1$-skeleton of a triangulation of $J$: A vertex of one triangle may lie on the side of another triangle.

We refine $\tilde \Gamma _n$ by adding  one  vertex $o_T$ inside of each triangle $T$ defined by
$\tilde \Gamma _n$ and by connecting $o_T$ by pairwise disjoint curves inside $T$ to
all vertices of $\tilde \Gamma _n$ on the sides of $T$. Denote the arising graph  by $\Gamma _n$.

By construction, $\Gamma _n$ defines a triangulation of $J$ and each triangle of the triangulation has diameter at most $\varepsilon$.

We now apply Proposition \ref{KeyLemma} and obtain a CAT$(\kappa)$ disc retract $W_n$, maps 
$p_n:J\to W_n$ and $q_n:W_n\to Y$ with the properties (1), (2), (3) stated there.

Hence, $p_n$ is surjective  and has contractible fibers, and $q_n$ is $1$-Lipschitz.

 Denote by $\mathcal V_n$ the vertices 
of the subgraph $\tilde \Gamma _n$ of $\Gamma _n$. 
Any vertex of $\Gamma _n$ lies at distance at most $\frac 1 n$ to some vertex of $\tilde \Gamma _n$. Since $f$ is $1$-Lipschitz, this distance estimate holds for the images of the vertices in $Y$ and, therefore for vertices of the triangulation of $W_n$.  Together with the property (2), this implies that $p_n(\mathcal V_n)$ is $\frac 2 n$-dense in $\mathcal W_n$.

For every pair of points $x,z\in \mathcal V_n$, we choose a curve $\gamma \subset \Gamma _n$ realizing the distance between $x$ and $z$ in $\Gamma _n$. Then
$$|x-z|_J + \frac 1 n \geq  \ell_J (\gamma) \geq \ell _Y (f\circ \gamma ) \geq \ell _W (p_n\circ \gamma) \geq |p_n(x)-p_n (z)|_W \,.$$

After choosing a subsequence we obtain a Gromov--Hausdorff converging sequence of images $p_n(\Gamma _n)$ to a compact metric space $W$.  
 Moreover, the spaces $W_n$ converge to the space $W$ as well and  the maps
$p_n$ converge to a surjective $1$-Lipschitz map $p:J\to W$. 
Since the spaces $W_n$ are CAT$(\kappa)$, the limit space $W$ is a CAT$(\kappa)$ space as well.

Set $A_n:=\mathcal V_n \cap G$. Then $A_n$ is $\frac 1 n$-dense in $G$ and, by Proposition \ref{KeyLemma}, the restriction of $q_n\circ p_n$ to $A_n$  coincides with $f|_{A_n}$. In the limit we obtain 
$$f|_{G} = q\circ p |_{G}\,.$$

 By assumption, $f$ preserves the length of all curves.  Since $q\circ p$ is $1$-Lipschitz,
we deduce $f\unrhd q\circ p$ (rel. $G$).   By assumption, $f$ is length-minimizing relative 
to $G$, hence $f=q\circ p$, by Lemma \ref{lem: coincide}.  Since $p$ and $q$ are $1$-Lipschitz, the map $p$ must preserve the length of any curve.

Since $f$ is a light map,  also   $p$ must  be a light map.  The fibers of $p_n$  converge to subsets of the corresponding fibers of $p$.   The fibers of $p_n$ are connected, hence any limit set is connected as well. Since all fibers of $p$ are totally disconnected, we deduce that
any limit of fibers $p_n^{-1} (w_n)$ must be a singleton. Denoting by $\varepsilon _n$ the 
maximal diameter of fibers of $p_n$, we deduce $\lim \varepsilon _n =0$.

We claim that $p$ is injective.  Otherwise, we find some $x\neq z \in p^{-1} (w)$.
Then $|p_n(x)-p_n(z)|_{W_n}$ converge to $0$.  Consider the geodesic $e_n$ 
between $p_n(x)$ and $p_n(z)$ in $W_n$ and set $E_n :=p^{-1}_n (e_n)$.
Since $p_n$ has connected fibers,  $E_n$ is a connected subset of $J$.

Taking a subsequence, we obtain in the limit a connected subset $E$ of $J$, which contains $x$ and $z$ and  is sent by 
$p$ to the point  $w$.  Since $p$ is a light map, this is a contradiction.

Therefore, $p:J\to W$ is injective, hence bijective. Since $J$ is compact, $p$ is a homeomorphism. Since $p$ preserves the length of any curve, $p$ is an isometry and
$J=W$ is a CAT$(\kappa)$ space.
\end{proof}

\section{Ruled discs}
We are going to prove Corollary \ref{cor: ruled} in this section. Thus, let  $\eta _0, \eta _1:[0,1]\to Y$ be rectifiable curve in a   CAT$(\kappa )$  space  $Y$. For $a\in [0,1]$, let $\gamma_a (t):[0,1]\to Y$ be geodesics  of length $<R_{\kappa}$ parametrized proportionally to arclength and connecting $\eta _{0} (a)$ with $\eta _1(a)$.

We consider the square $\mathcal Q:=[0,1]\times [0,1]$ as a topological disc. We define $f:\mathcal Q\to Y$ by
$f(a,t):=\gamma _a(t)$. We need  to verify that $\langle \mathcal Q \rangle _f$ is 
CAT$(\kappa)$.

Since geodesics of length $<R_{\kappa}$ depend continuously on the endpoints, we deduce
that $f$ is continuous.   By continuity,  we find some $\delta >0$, such that all geodesics
$\gamma _a$ have length at most $R_{\kappa} -\delta$.   By the quadrangle comparison, we find some $L, \rho >0$ with the following property:

Whenever $|\eta _i (a) -\eta _i (b)|  _Y \leq \rho$, for $i=0,1$, then, for all $t\in [0,1]$:
$$ |\gamma _a(t) -\gamma _b(t)|_Y  \leq L\cdot  \big(|\eta _0(a) -\eta _0(b)| _Y + |\eta _1(a) -\eta _1(b)|_Y\big)\;.$$

Therefore, for any $s\in[0,1]$, the curve $\eta _s (t):=f(t,s)$ is of finite length, bounded from above by $L\cdot (\ell _Y(\eta _0) +\ell _Y(\eta _1))$. Moreover, once $\eta _{0}, \eta _1 $ have 
length at most $r$ on an interval $[a,b] \subset [0,1]$, then  any of the horizontal curves 
$\eta _t, t\in [0,1]$ has length less than $2\cdot L\cdot r$.

Hence, any point  $(a,t)\in\mathcal Q$ sufficiently close to a given point $(a_0,t_0)$ can be connected to $(a_0,t_0)$ by a concatenation of a vertical and horizontal segments in $\mathcal Q$, which is mapped to a curve of a small length. Therefore, $f$ is
 length-continuous.

Therefore, Corollary \ref{cor: ruled} is a direct consequence of Theorem \ref{Theorem:main} and the following:

\begin{Lemma}
The map  $f:\mathcal Q \to Y$ is  length-minimizing relative to $\partial \mathcal Q$.
\end{Lemma}

\begin{proof}
Let a continuous map  $g:\mathcal Q\to Y$ satisfy $f\unrhd g$ (rel. $\partial \mathcal Q$).
 For the vertical segment
$t\to (a,t) \in \mathcal Q$, the image curve $t\to g(a,t)$
has length
not larger  than the geodesic $\gamma _a$ and connects the same boundary points.

Hence $t\to g(a,t)$ coincides with $\gamma _a$ up to parametrization.
Since also
$$\ell _Y(f\circ \gamma ) \geq \ell_Y(g\circ \gamma )$$
  for all subsegments $\gamma$ of the segment $t\to (a,t)$,
  parametrizations must coincide, hence $f(a,t)=g(a,t)$, for all $a$ and $t$.

Thus, $f\unrhd g$ (rel. $\partial \mathcal Q$)
implies $f=g$. Hence $f$ is  length-minimizing.  
\end{proof}

\section{Harmonic disks}
\subsection{Basics on Sobolev discs} In this section, we assume some knowledge on Sobolev and harmonic maps with values in metric spaces. We  refer  the reader to \cite{korevaar1993sobolev}, \cite{serbinowski1995harmonic} \cite{HKST}, \cite{LW}, for introductions to this subject.  

Throughout the section let $Y$ be a CAT$(\kappa)$ space and $\Omega$ be a bounded, open subset of $\mathbb R^2$.  Any Sobolev map $f\in W^{1,2} (\Omega , Y)$ has a (Korevaar--Schoen) energy $E(f)\in [0, \infty)$  \cite{korevaar1993sobolev}, \cite[Proposition 4.6]{LW}.

 If $\Omega$ is a Lipschitz domain, hence the boundary $\partial \Omega$ is a union of Lipschitz curves, 
then for any  Sobolev map $f\in W^{1,2} (\Omega , Y)$ there is a \emph{trace}
 of $f$, $tr(f)\in W^{1,2} (\partial \Omega, Y)$, \cite[Theorem 1.12.2]{korevaar1993sobolev}.  If $f: \bar \Omega \to Y$ is continuous, then the trace is just the restriction of $f$ to the boundary $\partial \Omega$.

If $\partial \Omega$ is not Lipschitz    we say  that  $f,g\in W^{1,2} (\Omega, Y)$ have \emph{equal traces},  if $\langle f-g \rangle _Y$ is contained in  $W^{1,2}_0 (\Omega , \mathbb R)$,  see  \cite[Section 1.4]{serbinowski1995harmonic}, \cite[Section 4.1]{LW-intrinsic}.

A map $f\in W^{1,2} (\Omega, Y)$ is called \emph{harmonic} if $f$ has the smallest energy among all maps $g\in W^{1,2} (\Omega, Y)$ with the same trace as $f$.

For any $f\in W^{1,2}(\Omega, B)$, where $B$ is a ball in $Y$ of radius $<\frac {R_{\kappa}} 2$, there exists a unique harmonic map $g\in W^{1,2} (\Omega, Y)$ with the same trace as $f$, \cite[Theorem 1.16]{serbinowski1995harmonic}. Moreover, the image of this  harmonic $g$ is contained in $B$ as well and $g$  has a (unique) locally Lipschitz representative 
\cite[Theorem 3.1]{serbinowski1995harmonic}.  We will always use this representative  below.  Finally, the restriction of a harmonic map  to any subdomain 
$\tilde \Omega \subset \Omega$ is harmonic as well,  \cite[Lemma 4.2]{LW-intrinsic}.

Assume that $f,g:\bar {\Omega} \to Y$ are continuous and that $f\in W^{1,2} (\Omega , Y)$. Assume further that $f\unrhd g$.
  Then $\ell_Y(f\circ \gamma)\geq \ell_Y(g\circ \gamma)$, for any  curve $\gamma$ in $ \Omega$.
Using the upper-gradient definition of Sobolev maps  \cite{HKST} this implies that $g \in W^{1,2}    (\Omega , Y)$  as well. Moreover, at almost all points of $\Omega$ the \emph{approximate metric differential} of $f$ is not less than the corresponding \emph{approximate metric differential} of $g$, \cite[Proposition 4.10]{LW}. Then
 also    the energies satisfy   $E(f)\geq E(g)$, \cite[Proposition 4.6]{LW}.   In particular, if $f$ is harmonic then $g$ is harmonic as well.

If $f\in   W^{1,2} (\mathbb D , Y)$ is harmonic and $tr(f):\mathbb S^1\to Y$ is continuous, then $f$ is continuous on the closed disc $\mathbb D$,
\cite[Proposition 4.4]{LW-energy}.
 The following  Lemma is essentially  contained in the proof of \cite[Proposition 4.4]{LW-energy}.

\begin{Lemma} \label{lem:cont}
Let $f:\mathbb D\to Y$ be a harmonic map such the trace  $f:\mathbb S^1 \to Y$ is a curve
of finite length.   Then $f$ is length-continuous.
\end{Lemma}

\begin{proof}
The map $f$ is locally Lipschitz on $\mathbb D_0:=\mathbb D\setminus \mathbb S^1$. Hence,
$\hat \pi _f:\mathbb D\to \langle \mathbb D \rangle _f$ is continuous on $\mathbb D_0$.

Let $z\in \mathbb S^1$  and  $\varepsilon >0$ be arbitrary. We find a small ball $B_r(z)$ such that
the lengths of the two curves  $f(\partial B_r(z) \cap \mathbb D)$ and $f(\mathbb S^1\cap B_r(z))$ are smaller than $\varepsilon$.  Moreover, if $r$ is small enough, the energy of the restriction of $f$ to $B_0 :=B_r(z) \cap \mathbb D$ is also smaller than $\varepsilon$.

Once $r$ and  $\varepsilon$ are sufficiently small, any point $p\in B_0 $ can be connected by a curve $\gamma$  with a point on $\partial B_0$ such that $\ell_Y(f\circ \gamma)$ is at most $\delta$, where $\delta$ goes to $0$ with $\varepsilon$, see \cite[Lemma 5.3]{LW-intrinsic}.

Therefore,  $\langle z -p \rangle _f \leq \delta +\varepsilon$. The right hand side   goes to $0$ as $r$ converges to $0$.
Hence the map $\hat \pi _f$ is continuous at $z$. Thus, $\hat \pi _f$ is continuous  on all of $\mathbb D$.
\end{proof}

It is possible to prove (but requires some rather technical considerations in the case of positive $\kappa$) that 
any harmonic disc $f:\mathbb D\to Y$ is length-minimizing.   We arrive at the proof of
Corollary \ref{cor:harm} faster restricting the domain of definition:

\begin{proof}[Proof of Corollary \ref{cor:harm}]
 Thus, let $f:\mathbb D\to Y$ be continuous and harmonic and such that $f(\mathbb S^1)$ has finite length. If, for some $y\in Y$, the set $\mathbb  D \setminus f^{-1} (y)$  has a connected component $O$ not intersecting $\mathbb S^1$, then the restriction of $f$ to this component $O$ could not be  harmonic.  
Hence, this is impossible.   Therefore,  the geodesic space $\langle \mathbb D \rangle _f$ is a disc retract by 
 Lemma \ref{lem:cont} and Lemma \ref{lem: banah}.


Let $G$ be a  Jordan triangle in $\langle \mathbb D \rangle _f$  of length less than $2R_{\kappa}$ and let $J$ be the closed disc bounded by $G$ in the disc retract $\langle \mathbb D \rangle _f$.  As in the proof of  Theorem \ref{Theorem:main}, it suffices to show that $\langle J\rangle _{\hat f}$ with the induced length metric is  CAT($\kappa$). Due to  Theorem \ref{Theorem:main}
it suffices to prove that $\hat f:J\to Y$ is length-minimizing relative to $G$.

Let $\tilde G$ be the preimage of $G$ in $\mathbb D$ and let $O\subset \mathbb D$ be the preimage of $J\setminus G$.   Then $f(\tilde G)=\hat f(G)$ is a curve of length less than $2R_{\kappa}$,
hence it is contained in a ball $B$ of radius less than $\frac {R_{\kappa}} 2$  in $Y$. 
Applying  a strictly $1$-Lipschitz retraction to $B$, \cite{lytchak2021short}, we deduce that the image of the harmonic  map $f:O\to Y$ is contained in $B$.

Assume now that $\hat  f$ is not length-minimizing relative to $G$  and let $g:J\to Y$ with $\hat  f\unrhd g$ (rel. $G$) be given.  Then the map $\tilde g :=g\circ \hat \pi _f :\bar O \to Y$  satisfies  $f \unrhd \tilde g$  (rel. $\tilde G$).  
  As seen above, $\tilde g$ is harmonic on $O$. By the uniqueness of harmonic maps with values in $B$, we deduce $\tilde  g=f$.  Hence $\hat  f=g$.
This finishes the proof of the fact  that $\hat f$ is length-minimizing and of the Corollary. 
\end{proof}

\section{Questions}

Here is the promised list of  (mostly technical) questions.  

\begin{Question}
What does it mean in differential-geometric terms for a smooth map $f:\mathbb D\to M$
into a Riemannian manifold $M$ to be length-minimizing? 
\end{Question}

See \cite[Section 10]{petrunin2019metric} for partial answers to this question.

\begin{Question}
Does our main theorem hold true for length-minimizing discs which are length-connected
but not length-continuous?
\end{Question}

This question was the essential motivation for the definition of metric-minimality created in 
\cite{petrunin2019metric}.

\begin{Question}
Does the conclusion of  Lemma \ref{lem: banah} hold true without the non-bubbling assumption?
For which topological spaces $X$ does the conclusion of Lemma \ref{lem: banah} hold true, for all length-continuous maps $f:X\to Y$?  In particular, does  Lemma \ref{lem: banah} hold true for Euclidean balls $X$ of dimension larger than $2$.
\end{Question}

The last part of this  question appears in the arXiv-version of \cite{Petruninint}.

\begin{Question}
Let  $f:\mathbb D\to Y$ be length-connected.  Can $f$ be not length-continuous if 
$\langle \mathbb D \rangle _f$  is compact?  What about more general Peano spaces $X$
instead of $\mathbb D$?
\end{Question}

\bibliographystyle{alpha}
\bibliography{lm}

\begin{thebibliography}{HKST15}

\bibitem[AKP23]{alexander2014alexandrov}
S~Alexander, V~Kapovitch, and A~Petrunin.
\newblock Alexandrov geometry.
\newblock {\em to appear, arxiv: 1903.08539}, 2023.

\bibitem[Ale57]{alexandrov1957uber}
A.~D. Alexandrov.
\newblock {Über eine Verallgemeinerung der Riemannschen Geometrie}.
\newblock {\em {Schriftenreiche der Institut für Mathematik}}, 1:33--84, 1957.

\bibitem[Bal12]{ballmann2012lectures}
W.~Ballmann.
\newblock {\em Lectures on spaces of nonpositive curvature}, volume~25.
\newblock Birkh{\"a}user, 2012.

\bibitem[BBI01]{burago2022course}
D.~Burago, Y.~Burago, and S.~Ivanov.
\newblock {\em A course in metric geometry}, volume~33.
\newblock American Mathematical Society, 2001.

\bibitem[BH99]{Bridson}
M.~Bridson and A.~Haefliger.
\newblock {\em Metric spaces of non-positive curvature}, volume 319 of {\em
  Grundlehren der mathematischen Wissenschaften}.
\newblock Springer-Verlag, Berlin, 1999.

\bibitem[Bow95]{bowditch1995notes}
B.~Bowditch.
\newblock Notes on locally {CAT}(1) spaces.
\newblock {\em Geometric group theory (Columbus, OH, 1992)}, 3:1--48, 1995.

\bibitem[CR22]{CR}
P.~Creutz and M.~Romney.
\newblock Triangulating metric surfaces.
\newblock {\em Proc. Lond. Math. Soc. (3)}, 125(6):1426--1451, 2022.

\bibitem[Cre21]{creutz2021space}
P.~Creutz.
\newblock Space of minimal discs and its compactification.
\newblock {\em Geometriae Dedicata}, 210(1):151--164, 2021.

\bibitem[HKST15]{HKST}
J.~Heinonen, P.~Koskela, N.~Shanmugalingam, and J.~Tyson.
\newblock {\em Sobolev spaces on metric measure spaces}, volume~27 of {\em New
  Mathematical Monographs}.
\newblock Cambridge University Press, Cambridge, 2015.
\newblock An approach based on upper gradients.

\bibitem[Jos83]{Jost}
J.~Jost.
\newblock Existence proofs for harmonic mappings with the help of a maximum
  principle.
\newblock {\em Math. Z.}, 184(4):489--496, 1983.

\bibitem[KS93]{korevaar1993sobolev}
N.~Korevaar and R.~Schoen.
\newblock Sobolev spaces and harmonic maps for metric space targets.
\newblock {\em Comm. Anal. Geom.}, 1(4):561--659, 1993.

\bibitem[LP21]{lytchak2021short}
A.~Lytchak and A.~Petrunin.
\newblock Short retractions of {CAT}(1) spaces.
\newblock {\em Proc. Amer. Math. Soc.}, 149(3):1247--1257, 2021.

\bibitem[LS97]{Lang}
U.~Lang and V.~Schroeder.
\newblock Kirszbraun's theorem and metric spaces of bounded curvature.
\newblock {\em Geom. Funct. Anal.}, 7(3):535--560, 1997.

\bibitem[LW16]{LW-energy}
A.~Lytchak and S.~Wenger.
\newblock Regularity of harmonic discs in spaces with quadratic isoperimetric
  inequality.
\newblock {\em Calc. Var. Partial Differential Equations}, 55(4):Art. 98, 19,
  2016.

\bibitem[LW17]{LW}
A.~Lytchak and S.~Wenger.
\newblock Area minimizing discs in metric spaces.
\newblock {\em Arch. Ration. Mech. Anal.}, 223(3):1123--1182, 2017.

\bibitem[LW18a]{LW-intrinsic}
A.~Lytchak and S.~Wenger.
\newblock Intrinsic structure of minimal discs in metric spaces.
\newblock {\em Geom. Topol.}, 22(1):591--644, 2018.

\bibitem[LW18b]{LW-isop}
A.~Lytchak and S.~Wenger.
\newblock Isoperimetric characterization of upper curvature bounds.
\newblock {\em Acta Math.}, 221(1):159--202, 2018.

\bibitem[LW20]{LW-param}
A.~Lytchak and S.~Wenger.
\newblock Canonical parameterizations of metric disks.
\newblock {\em Duke Math. J.}, 169(4):761--797, 2020.

\bibitem[Mes01]{Mese}
C.~Mese.
\newblock The curvature of minimal surfaces in singular spaces.
\newblock {\em Comm. Anal. Geom.}, 9(1):3--34, 2001.

\bibitem[Moo25]{Moore}
R.~Moore.
\newblock Concrning upper semi-continuous collections of continua.
\newblock {\em Trans. Amer. Math. Soc,}, 27:416--428, 1925.

\bibitem[Mor35]{Morrey}
C.~Morrey, Jr.
\newblock The {T}opology of ({P}ath) {S}urfaces.
\newblock {\em Amer. J. Math.}, 57(1):17--50, 1935.

\bibitem[NR21]{NR}
D.~Ntalampekos and M.~Romney.
\newblock Polyhedral approximations of metric surfaces and applications to
  uniformization.
\newblock {\em arXiv preprint arXiv:2107.07422}, 2021.

\bibitem[NSY21]{nagano2021two}
K.~Nagano, T.~Shioya, and T.~Yamaguchi.
\newblock Two-dimensional metric spaces with curvature bounded above i.
\newblock {\em Geom. Topol., to appear}, 2021.

\bibitem[Pet99]{Petold}
A.~Petrunin.
\newblock Metric minimizing surfaces.
\newblock {\em Electron. Res. Announc. Amer. Math. Soc.}, 5:47--54, 1999.

\bibitem[Pet10]{Petruninint}
A.~Petrunin.
\newblock Intrinsic isometries in {E}uclidean space.
\newblock {\em Algebra i Analiz}, 22(5):140--153, 2010.

\bibitem[Pet22]{petruninmetric}
A~Petrunin.
\newblock Pure metric geometry: introductory lectures.
\newblock {\em Preprint, arxiv.org/abs/2007.09846}, 2022.

\bibitem[PS19a]{petrunin2019metric}
A.~Petrunin and S.~Stadler.
\newblock Metric-minimizing surfaces revisited.
\newblock {\em Geom. Topol.}, 23(6):3111--3139, 2019.

\bibitem[PS19b]{PS-2}
A.~Petrunin and S.~Stadler.
\newblock Monotonicity of saddle maps.
\newblock {\em Geom. Dedicata}, 198:181--188, 2019.

\bibitem[Ser95]{serbinowski1995harmonic}
T.~Serbinowski.
\newblock {\em Harmonic maps into metric spaces with curvature bounded above}.
\newblock The University of Utah, 1995.

\end{thebibliography}

\end{document}